\documentclass[11pt]{article}
\usepackage[]{amsmath,amssymb}
\usepackage[]{amsthm,enumerate}
\usepackage{verbatim,bm}
\usepackage{color}
\usepackage[]{comment}
\usepackage{a4wide}
\usepackage{lscape}
\usepackage{arydshln,comment}

\newtheorem{theorem}{Theorem}[section]
\newtheorem{proposition}[theorem]{Proposition}
\newtheorem{lemma}[theorem]{Lemma}
\newtheorem{corollary}[theorem]{Corollary}
\theoremstyle{definition}
\newtheorem{definition}[theorem]{Definition}
\newtheorem{example}[theorem]{Example}
\newcommand{\defn}[1]{{\em #1}}
\theoremstyle{remark}
\newtheorem{remark}[theorem]{Remark}

\makeatletter
\@addtoreset{equation}{section}

\makeatother

\newcommand{\Jq}{\J^{(q)}}
\newcommand{\Aq}{\A^{(q)}}

\newcommand{\A}{\mathcal{A}}

\newcommand{\J}{\mathcal{J}}

\title{Balancedly splittable Hadamard matrices} 
\date{
\today
}

\author{
 Hadi Kharaghani\thanks{Department of Mathematics and Computer Science, University of Lethbridge,
Lethbridge, Alberta, T1K 3M4, Canada. \texttt{kharaghani@uleth.ca}} 
\and  
 Sho Suda\thanks{Department of Mathematics Education,  Aichi University of Education, 1 Hirosawa, Igaya-cho,  Kariya, Aichi, 448-8542, Japan. \texttt{suda@auecc.aichi-edu.ac.jp}}
}

\begin{document}
\maketitle
\begin{abstract}
Balancedly splittable Hadamard matrices are introduced and studied. A 
connection is made to the Hadamard diagonalizable strongly regular 
graphs, maximal equiangular lines set, 
 and unbiased Hadamard matrices. 
Several construction methods are presented. As an application, commutative
 association schemes of 4, 5, and 6 classes are constructed.

\end{abstract}

\section{Introduction}
An $n\times n$ matrix $H$ is a \defn{Hadamard matrix of order $n$} if its entries are $1,-1$ and it satisfies $HH^\top=I_n$, where $I_n$ denotes the identity matrix of order $n$.  
Hadamard matrices $H$ are shown to be related to other combinatorial objects such as combinatorial designs, distance regular graphs of diameter $4$, and the following under some regularity conditions: 
\begin{itemize}
\item strongly regular graphs if $H$ is symmetric with constant diagonal \cite{GS}, 
\item doubly regular tournaments if $H$ is skew-symmetric \cite{RB1972},
\item symmetric or non-symmetric association schemes with $3$ classes if $H$ is of symmetric or skew-symmetric Bush-type \cite{GC}.
\end{itemize}

A Hadamard matrix $H$ of order $n$ is said to be \defn{balancedly splittable} if there is an $\ell\times n$ submatrix $H_1$ of $H$ such that inner products for any two distinct column vectors of $H_1$ take at most two values. 
More precisely, there exist integers $a,b$ and the adjacency matrix $A$ of a graph such that $H_1^\top H_1=\ell I_n +a A+b(J_n-A-I_n)$, where $J_n$ denotes the all-ones matrix of order $n$. 

We will show that the matrix $A$ is (switching equivalent)to the adjacency matrix of a strongly regular graph, and the graph is Hadamard diagonalizable in the sense of Barik, Fallat, and Kirkland \cite{BFK}.  
The case $b=-a$ corresponds to a maximal equiangular lines set, 
and to the unbiased Hadamard matrices.  
We propose various constructions and provide non-existence results for balancedly splittable Hadamard matrices. 
A construction provided in Section~\ref{sec:constq(q+1)} is related to recent work on complex Hadamard matrices \cite{FKS} and the submatrix $H_1$ of a balancedly Hadamard matrix is essentially the same as the quasi-symmetric design constructed in \cite{HKSane}. 
It turns out that the property of balancedly splittable leads to a new relation between Hadamard matrices and various combinatorial objects.

As a further application we will make a connection to association schemes and provide schemes with $4,5,6$-classes by using balancedly splittable Hadamard matrices and Latin squares.  
We will demonstrate  that our approach relates to different concepts presented in \cite{HKO,HKS,KSS,KS,KSpre}, particularly in constructing unbiased Hadamard matrices and association schemes.


\section{Balancedly splittable Hadamard matrices}
\begin{definition}\label{def:bsh}
A Hadamard matrix $H$ of order $n$ is \defn{balancedly splittable} if 
by suitably permuting its rows it can be transformed to $H=\begin{pmatrix}H_1\\H_2\end{pmatrix}$ 
such that the matrix $H_1^\top H_1$ has at most two distinct off-diagonal entries. 
In this case we say that  $H$ is balancedly splittable with respect to $H_1$. 
\end{definition}
Let $H_1$ be an $\ell\times n$ matrix. 
Then there exist integers $a,b$ and a $(0,1)$-matrix $A$ such that $a\geq b$ and 
\begin{align}\label{eq:bsh}
H_1^\top H_1=\ell I_n+a A+b(J_n-A-I_n).
\end{align} 
The tuple of values $(n,\ell,a,b)$ is said to be the parameters of a balancedly splittable Hadamard matrix of order $n$ with respect to $H_1$. 

By the equation $H_1^\top H_1+H_2^\top H_2=nI_n$, we have the following lemma. 
\begin{lemma}
Let $H=\begin{pmatrix}H_1\\H_2\end{pmatrix}$ be a Hadamard matrix of order $n$ with $\ell\times n$ matrix $H_1$, $1\leq\ell <n$. 
Then $H$ is balancedly splittable with the parameters $(n,\ell,a,b)$ with respect to $H_1$ if and only if $H$ is balancedly splittable with the parameters $(n,n-\ell,-b,-a)$ with respect to $H_2$. 
\end{lemma} 
The following are examples with $a=b$. 
\begin{example}
\begin{enumerate}
\item Any Hadamard matrix is balancedly splittable with respect to itself with the parameters $(n,n,0,0)$. 
\item A Hadamard matrix $H=\begin{pmatrix}H_1\\H_2\end{pmatrix}$ of order $n$ with $H_1$ the all-ones vector  is balancedly splittable with respect to $H_1$ with the parameters $(n,1,1,1)$. 
\end{enumerate}
\end{example}
Conversely, it is easy to characterize a balancedly splittable Hadamard matrix to have $H_1^\top H_1$ with the only one distinct off-diagonal entry, as shown below. 
\begin{proposition}\label{prop:bsh1}
If a Hadamard matrix $H$ is balancedly splittable with $H_1^\top H_1=\ell I_n+a (J_n-I_n)$, then $(\ell,a)\in\{(1,1),(n-1,-1),(n,0)\}$.
\end{proposition}
\begin{proof}
Squaring the equation $H_1^\top H_1=\ell I_n+a (J_n-I_n)$ 
yields that $n(\ell-a)I_n+n a J_n=(\ell-a)^2I_n+(2(\ell-a)a+a^2 n)J_n$.  
Comparing coefficients with $1\leq \ell\leq n$, we have that $(\ell,a)\in\{(1,1),(n-1,-1),(n,0)\}$. 
\end{proof}

In the rest of the paper, we focus on balancedly splittable Hadamard matrices $H$ such that $H_1^\top H_1$ has exactly two distinct values off diagonal. 
The following is  an obvious example. 
\begin{example}
Let $H=\begin{pmatrix}H_1\\H_2\end{pmatrix}$ be a Hadamard matrix of order $n$ with $1\times n$ matrix $H_1$. 
If $H_1$ is not equal to the all-ones vector, then $H$ is balancedly splittable with respect to $H_1$ with the parameters $(n,1,1,-1)$.     
\end{example}
Throughout the rest of the paper, we assume that $1<\ell<n-1$ in order to avoid the trivial cases.

A \defn{strongly regular graph} with parameters $(v,k,\lambda,\mu)$ is a regular graph with $v$ vertices and degree $k$ such that every two adjacent (non-adjacent resp.) vertices have $\lambda$ ($\mu$ resp.) common neighbors. 
The \defn{Seidel matrix} of a graph with adjacency matrix $A$ is $S=J_v-I_v-2A$. 
A \defn{strong graph} is such that its Seidel matrix $S$ satisfies the property that $S^2$ is a linear combination of $S,I_v,J_v$.  
It is known that a graph is strongly regular if and only if it is regular and strong, see for \cite{S} and \cite[Chapter 10]{BH}.   

Balancedely splittable Hadamard matrices are related to strong graphs and strongly regular graphs, as shown in the following proposition.  
\begin{proposition}\label{prop:bshp}
Let $H=\begin{pmatrix}H_1\\H_2\end{pmatrix}$ be a balancedly splittable Hadamard matrix of order $n$ with $H_1^\top H_1=\ell I_n+a A+b(J_n-A-I_n)$ where $A$ is an $n\times n$ $(0,1)$-matrix, and $1<\ell<n-1$, $b<a$. 
\begin{enumerate}[(1)]
\item If $b= -a$, $S=J_n-I_n-2A$ is the Seidel matrix of a strong graph satisfying  $S^2=\frac{n-2\ell}{a}S+\frac{\ell(n-\ell)}{a^2}I_n$ and $n=\frac{\ell^2-a^2}{\ell-a^2}$. 
The strong graph is switching equivalent to a strongly regular graph with the parameters $(n,k,\lambda,\mu)$ where $k,\lambda,\mu$ are the following: 
\begin{align}\label{eq:bshp7}
k=\frac{(a-1) \ell (a+\ell)}{2 a(\ell-a^2)},\quad 
\lambda=\frac{(a+\ell) (3 a^2+a\ell-a-3 \ell)}{4 a(\ell-a^2)},\quad 
\mu=\frac{(a-1) \left(\ell^2-a^2\right)}{4 a(\ell-a^2)}.
\end{align}

\item If $b\neq -a$, 
then $A$ is the adjacency matrix of a strongly regular graph with parameters $(n,k,\lambda,\mu)$, where $b,k,\lambda,\mu$ are either one of the following:
\begin{enumerate}[(a)]
\item $b=\frac{\ell (-a+\ell-n)}{a (n-1)+\ell}$, $k=\frac{\ell n (n-\ell-1)}{n (a^2+\ell)-(a-\ell)^2}$, $\lambda=\frac{n (n^2 (a^3+\ell^2)-2 (\ell+1) n (a^3+\ell^2)+(2 a \ell+a+\ell (\ell+2)) (a-\ell)^2)}{((a-\ell)^2-n (a^2+\ell))^2}$, $\mu=\frac{\ell n (a-\ell) (\ell-n+1) (a-\ell+n)}{((a-\ell)^2-n (a^2+\ell))^2}$. 
\item $b=\frac{(a-\ell) (\ell-n)}{a (n-1)+\ell-n},k=\frac{(\ell-1) n (\ell-n)}{(a-\ell)^2-n ((a-2) a+\ell)}$, $\lambda=\frac{n (a^3 (-2 \ell (n-1)+n^2-1)-3 a^2 (\ell-n)^2+3 a (\ell-n)^2+(\ell-2) \ell (\ell-n)^2)}{((a-\ell)^2-n ((a-2) a+\ell))^2}$, $\mu=\frac{(\ell-1) n (a-\ell) (\ell-n) (a-\ell+n)}{((a-\ell)^2-n ((a-2) a+\ell))^2}$. 
\end{enumerate}
Furthermore if (a) occurs, then each row of $H_1$ is orthogonal to the all-ones vector. 
\end{enumerate}
\end{proposition}
\begin{proof}
Squaring \eqref{eq:bsh}
 with the fact that $H_1H_1^\top=nI_{\ell}$, we have that 
\begin{align}\label{eq:bshp}
(\ell I_n+a A+b(J_n-A-I_n))^2&=n(\ell I_n+a A+b(J_n-A-I_n)).
\end{align}
Simplifying \eqref{eq:bshp} by $b\neq a$ yields that 
\begin{align}\label{eq:bshp6}
A^2=\frac{1}{(a-b)^2}\big(&(a-b)(n-2\ell+2b)A\nonumber\\
&+(\ell-b)(n-\ell +b)I_n+b(n-nb-2\ell+2b)J_n-(a-b)b (AJ_n+J_nA)\big). 
\end{align}

For $x\in\{1,\ldots,n\}$, let $k_x$ denote the degree of $x$ in the graph whose adjacency matrix is $A$.  
Then comparing the $(x,x)$-entry in \eqref{eq:bshp} shows that 
\begin{align}\label{eq:bshp1}
\ell^2+a^2 k_x+b^2(n-1-k_x)=n \ell.
\end{align} 

(1): For the case $b=-a$, by $H_1^\top H_1=\ell I_n+a S$, \eqref{eq:bshp} is reduced to 
$S^2=\frac{n-2\ell}{a}S+\frac{\ell(n-\ell)}{a^2}I_n$, and \eqref{eq:bshp1} shows that $\ell^2+a^2(n-1)=n\ell$. Since $\ell \neq 1$, we have that $\ell\neq a^2$.  Thus $n=\frac{\ell^2-a^2}{\ell-a^2}$. 

Normalize the Hadamard matrix $H$ so that the last row of $H$ equals to the all-ones vector. Then multiplying $J_n$ by $H_1^\top H_1=(\ell+a) I_n+2a A-aJ_n$, we have $2aAJ_n=(an-\ell-a)J_n$. 
Since $a\neq 0$, the graph is regular with valency $k=\frac{an-\ell-a}{2a}$. 
 The strong graph with the Seidel matrix $S$ is regular, and thus it is strongly regular. Let  $(n,k,\lambda,\mu)$ be its  parameters.   
The parameters are determined as in \eqref {eq:bshp7} 
 by substituting $b=-a$ and $AJ_n=J_nA=\frac{an-\ell-a}{2a}J_n$ into \eqref{eq:bshp6} with use of $n=\frac{\ell^2-a^2}{\ell-a^2}$.

(2): By the assumption that $b\neq \pm a$, \eqref{eq:bshp1} shows that 
$k_x=\frac{n\ell-\ell^2-b^2(n-1)}{a^2-b^2}$,  
which is independent of the particular choice of $x$. Thus $A$ is the adjacency matrix of a regular graph of degree $k$ given as
\begin{align}\label{eq:bshp5}
k:=\frac{n\ell-\ell^2-b^2(n-1)}{a^2-b^2}.
\end{align} 

To use the fact that $AJ_n=J_nA=kJ_n$, \eqref{eq:bshp6} shows that 
the matrix $A$ is the adjacency matrix of a strongly regular graph with parameters $(n,k,\lambda,\mu)$ where $\lambda,\mu$ are determined as follows:
\begin{align}
\lambda&=\frac{n (a^2-a (b-1) b+b^3-2 b \ell)+2 (b-\ell)(a^2+a b-b (b+\ell))}{(a-b)^2 (a+b)},\label{eq:bshp3}\\
\mu&=\frac{b n (-a b+a+b^2+b-2 \ell)+2 b (a-\ell) (b-\ell)}{(a-b)^2 (a+b)}.\label{eq:bshp4}
\end{align}
Substituting \eqref{eq:bshp3}, \eqref{eq:bshp4} into the well-known formula $k(k-\lambda-1)=(n-k-1)\mu$ and simplifying it, we have
\begin{align}\label{eq:bshp10}
((a (n-1)+\ell)b-(-a+\ell-n)) ((a (n-1)+\ell-n)b-(a-\ell) (\ell-n))=0. 
\end{align}
Since $a(n-1)+\ell>0$ and $a (n-1)+\ell-n>0$, \eqref{eq:bshp10} implies that $b=\frac{\ell (-a+\ell-n)}{a (n-1)+\ell}$ or $b=\frac{(a-\ell) (\ell-n)}{a (n-1)+\ell-n}$.
Thus by \eqref{eq:bshp5}, \eqref{eq:bshp3}, \eqref{eq:bshp4} we obtain the desired formula for $k,\lambda$ and $\mu$. 

For (a),  
pre-multiplying the all-ones column vector ${\bf 1}$ and post-multiplying its transpose by \eqref{eq:bsh} shows that 
\begin{align*}
(H_1{\bf 1})^\top(H_1{\bf 1})={\bf 1}^\top (\ell I_n+a A+b(J_n-A-I_n)){\bf 1}=(\ell+a k+b (n-1- k))n=0, 
\end{align*}
where we used $b=\frac{\ell (-a+\ell-n)}{a (n-1)+\ell}$, $k=\frac{\ell n (n-\ell-1)}{n (a^2+\ell)-(a-\ell)^2}$ in the last equality. 
Thus $H_1 {\bf 1}={\bf 0}$ holds where  ${\bf 0}$ is the zero vector. 
\end{proof}
\begin{remark}
It is routinely checked that the parameters of the strongly regular graphs in (2) (a) are valid for the case (1) to use $\frac{\ell(n-\ell)}{a^2}=n-1$. 
\end{remark}

\begin{remark}\label{rem:sreg}
Let $H=\begin{pmatrix}H_1\\H_2\end{pmatrix}$ be a Hadamard matrix of order $n$.
\begin{enumerate}
\item  It is easy to check that $H$ is balancedly splittable with the parameters $(n,\ell,a,b)$ with respect to $H_1$ fitting into Proposition~\ref{prop:bshp}(2)(a) if and only if $H$ is balancedly splittable with the parameters $(n,n-\ell,-b,-a)$ with respect to $H_2$ fitting into Proposition~\ref{prop:bshp}(2)(b). 
\item Assume that $H$ is a balancedly splittable Hadamard matrix of order $n$ with the parameters $(n,\ell,a,-a)$ with respect to $H_1$ and the last row being the all-ones vector. Then  
the Hadamard matrix $H$ is a balancedly splittable Hadamard matrix of order $n$ with the parameters $(n,n-\ell-1,a-1,-a-1)$ with respect to the submatrix of $H$ obtained by deleting $H_1$ and the last row from $H$. 
\end{enumerate}
\end{remark}

\begin{remark}\label{rem:imp}
A strongly regular graph is said to be \defn{imprimitive} if either the graph or its complement is disconnected. This is equivalent to $k=\lambda+1$ or $k=\mu$. 
The former occurs in (2)(a) if and only if $a=\ell$.  
The latter occurs in (2)(a) if and only if $a=0$.
\end{remark}

A graph is said to be \defn{Hadamard diagonalizable} if its Laplacian matrix $L$ is diagonalized by a Hadamard matrix, that is, there exists a Hadamard matrix $H$ such that $HLH^\top$ is a diagonal matrix \cite{BFK}. 
It turns out that a Hadamard diagonalizable graph is regular \cite[Theorem~5]{BFK}. 
Therefore a graph is Hadamard diagonalizable if and only if its adjacency matrix is diagonalized by a Hadamard matrix.   

\begin{corollary}
\begin{enumerate}\item 
If a Hadamard matrix $H=\begin{pmatrix}H_1\\H_2\end{pmatrix}$ of order $n$ is balancedly splittable with respect to $H_1$ with parameters $(n,\ell,a,b)$ such that $H_2$ has the all-ones row vector, then the strongly regular graph constructed in Theorem~\ref{prop:bshp} is Hadamard diagonalizable by $H$.  
\item Conversely, if a strongly regular graph on $n$ vertices which is Hadamard diagonalizable by a normalized Hadamard matrix $H$, then $H$ is balancedly splittable with parameters $(n,\ell,a,b)$ where 
\end{enumerate}
\end{corollary}
\begin{proof}
Assume (1) to be true. 
It holds that 
\begin{align*}
H H_1^\top H_1 H^\top=\text{diag}(\underbrace{n^2,\ldots,n^2}_{\ell},\underbrace{0,\ldots,0}_{n-\ell}). 
\end{align*}
Without loss of generality, we may assume that the last row of $H_2$ is the all-ones vector. 
Then we have that $H J_n H^\top=\text{diag}(0,\ldots,0,n^2)$. 
Pre-multiplying $H$ and post-multiplying $H^\top$ by \eqref{eq:bsh} and simplifying it yields that 
$$
H A H^\top=\frac{n}{a-b}\text{diag}(\underbrace{-\ell+b+n,\ldots,-\ell+b+n}_{\ell},\underbrace{-\ell+b,\ldots,-\ell+b}_{n-\ell-1},-\ell-b(n-1)). 
$$
Therefore $A$ is Hadamard diagonalizable by $H$. 

Conversely assume (2), namely let $A$ be a strongly regular graph which is diagonalized by a normalized Hadamard matrix $H$ of order $n$. 
Without loss of generality $H$ has the all-ones vector as the last row. 
Then it holds that by a suitable rearranging rows and columns of $A$,
\begin{align}\label{eq:hah}
HAH^\top=\text{diag}(\underbrace{\theta,\ldots,\theta}_{\ell},\underbrace{\tau,\ldots,\tau}_{n-\ell-1},k)
\end{align} 
where $k$ is the valency of $A$, and $\theta,\tau$ are distinct eigenvalues of $A$ (one of which may be equal to $k$) , and  $\ell$ is the multiplicity of $\theta$. 
Write $H=\begin{pmatrix} H_1\\H_2\end{pmatrix}=\begin{pmatrix} H_1\\H_2' \\ {\bf 1}^\top\end{pmatrix}$ where $H_1$ is the $\ell\times n$ matrix and $H_2$ is the $(n-\ell)\times n$ matrix and $H_2'$ is the $(n-\ell-1)\times n$ submatrix of $H_2$.  
Pre-multiplying $H^\top$ and post-multiplying $H$ by \eqref{eq:hah} provides 
\begin{align}\label{eq:hah1}
n^2A=\theta H_1^\top H_1+\tau (H_2')^\top H_2'+k J_n.
\end{align} 
On the other hand, by $H^\top H=nI$, we have 
\begin{align}\label{eq:hah2}
nI_n=H_1^\top H_1+(H_2')^\top H_2'+J_n. 
\end{align}
Therefore by \eqref{eq:hah1} and \eqref{eq:hah2} we have that 
\begin{align*}
H_1^\top H_1=\frac{1}{\theta-\tau}(n^2A-\tau nI_n-(k-\tau)J_n). 
\end{align*}
Thus a Hadamard matrix $H=\begin{pmatrix} H_1\\H_2\end{pmatrix}$ is balancedly splittable with respect to the $\ell\times n$ matrix $H_1$ such that $H_2$ has the all-ones row vector. 
\end{proof}


We list the feasible parameters in Table~\ref{tab:1} for (1) with $n\leq 1024,\ell\leq n/2$
 and those in Table~\ref{tab:2} for (2)(a) with $n\leq 64$ and $0<a<\ell$. 
In the tables, E stands for ``exists" and NE stands for ``does not exist". 

\begin{table}[htbp]
  \begin{center}
          \caption{$n\leq 1024,\ell\leq n/2$}\label{tab:1}
\begin{tabular}{cccc|cccc}
$n$ & $\ell$ & $a$ & & $n$ & $\ell$ & $a$ &\\ \hline
 16 & 6 & 2 & \text{E} & 528 & 187 & 11 & \text{NE, Prop~\ref{prop:non-exist1}}\\
 36 & 15 & 3 & \text{NE, Prop~\ref{prop:non-exist1}} & 540 & 99 & 9 & \text{NE, Prop~\ref{prop:non-exist2}}\\
 64 & 28 & 4 & \text{E}  & 560 & 130 & 10 & \\
 100 & 45 & 5 & \text{NE, Prop~\ref{prop:non-exist1}} & 576 & 276 & 12 & \\
 120 & 35 & 5 & \text{NE, Prop~\ref{prop:non-exist2}} & 616 & 165 & 11 & \text{NE, Prop~\ref{prop:non-exist2}}\\
 144 & 66 & 6 &  & 640 & 72 & 8 & \\
 196 & 91 & 7 & \text{NE, Prop~\ref{prop:non-exist1}} & 676 & 325 & 13 & \text{NE, Prop~\ref{prop:non-exist1}}\\
 256 & 120 & 8 & \text{E} & 780 & 247 & 13 & \text{NE, Prop~\ref{prop:non-exist2}}\\
 280 & 63 & 7 & \text{NE, Prop~\ref{prop:non-exist1}} & 784 & 378 & 14 & \\
 288 & 42 & 6 &  & 900 & 435 & 15 & \text{NE, Prop~\ref{prop:non-exist1}}\\
 320 & 88 & 8 &  & 924 & 143 & 11 & \text{NE, Prop~\ref{prop:non-exist1}}\\
 324 & 153 & 9 & \text{NE, Prop~\ref{prop:non-exist1}} & 936 & 221 & 13 & \text{NE, Prop~\ref{prop:non-exist1}}\\
 400 & 190 & 10 &  & 1008 & 266 & 14 & \\
 484 & 231 & 11 & \text{NE, Prop~\ref{prop:non-exist1}} & 1024 & 496 & 16 & \text{E} \\
\end{tabular}
  \end{center}
\end{table}
\begin{table}[htbp]
  \begin{center}
          \caption{$n\leq 64$ and $0<a<\ell$}\label{tab:2}
         \begin{tabular}{cccccccc
}
 $n$ & $\ell$ & $a$ & $b$ & $k$ & $\lambda$ & $\mu$ & 
\\ \hline
 16 & 5 & 1 & $-3$ & 10 & 6 & 6 & E, Remark~\ref{rem:sreg}(2) 
\\
 16 & 9 & 1 & $-3$ & 9 & 4 & 6 & E, Remark~\ref{rem:sreg}(2) 
\\
 36 & 10 & 4 & $-2$ & 10 & 4 & 2 & \text{NE, Prop~\ref{prop:ne36} (1)} \\
 36 & 14 & 2 & $-4$ & 21 & 12 & 12 & \text{NE, Prop~\ref{prop:ne36} (3)} \\
 36 & 20 & 2 & $-4$ & 20 & 10 & 12 & \text{NE, Prop~\ref{prop:ne36} (4)} \\
 36 & 25 & 1 & $-5$ & 25 & 16 & 20 & \text{NE, Prop~\ref{prop:ne36} (2)} \\
 64 & 14 & 6 & $-2$ & 14 & 6 & 2 & E, Theorem~\ref{thm:31} \\
 64 & 18 & 2 & $-6$ & 45 & 32 & 30 & \\
 64 & 21 & 5 & $-3$ & 21 & 8 & 6 &  \\
 64 & 27 & 3 & $-5$ & 36 & 20 & 20 & E, Remark~\ref{rem:sreg}(2) \\
 64 & 35 & 3 & $-5$ & 35 & 18 & 20 & E, Remark~\ref{rem:sreg}(2) \\
 64 & 42 & 2 & $-6$ & 42 & 26 & 30 &  \\
 64 & 45 & 5 & $-3$ & 18 & 2 & 6 &  \\
 64 & 49 & 1 & $-7$ & 49 & 36 & 42 & E, Theorem~\ref{thm:31} \\
    \end{tabular}
  \end{center}
\end{table}


The following upper bound is due to Delsarte, Goethals and Seidel \cite{DGS}.  The finite set $X$ of $\mathbb{R}^m$ satisfying the assumption in Proposition~\ref{prop:ub} is referred to as an \defn{equiangular lines set}. 

\begin{proposition}\label{prop:ub}
Let $X\subset \mathbb{R}^m$ be a set of unit vectors such that $|\langle v,w \rangle|=\alpha$ for all $v,w\in X,v\neq w$. 
If $m<\frac{1}{\alpha^2}$, then 
\begin{align}\label{ineq:ub}
|X|\leq \frac{m(1-\alpha^2)}{1-m \alpha^2}.
\end{align} 
\end{proposition}
\begin{proposition}
If there exists a balancedly splittable Hadamard matrix with the parameters $(n,\ell,a,-a)$, then there exists an equiangular lines set in $\mathbb{R}^{\ell}$ with inner product $\frac{\sqrt{n-\ell}}{\sqrt{\ell(n-1)}}$ attaining equality in \eqref{ineq:ub}. 
\end{proposition}
\begin{proof}
Let $H=\begin{pmatrix} H_1\\H_2\end{pmatrix}$ be a balancedly Hadamard matrix with respect to an $\ell\times n$ matrix $H_1$. 
Let $X$ be the set of column vectors of $H_1$ normalized by dividing by $\sqrt{\ell}$, so $a^2\ell^2=\frac{n\ell-\ell^2}{n-1}$.
Then $X$  is a subset of $n$ unit vectors of $\mathbb{R}^\ell$ such that $|\langle v,w \rangle|=\frac{\sqrt{n-\ell}}{\sqrt{\ell(n-1)}}$ for all $v,w\in X, v\neq w$.  
It can be seen  that $m<\frac{1}{\alpha^2}$ for $(m,\alpha)=(\ell,\frac{\sqrt{n-\ell}}{\sqrt{\ell(n-1)}})$ and that the right hand side in \eqref{ineq:ub} in this case equals to $n$. Thus our equiangular lines set attains the bound in \eqref{ineq:ub}.  
\end{proof}

Two Hadamard matrices $H$ and $K$ of order $n$ are said to be \defn{unbiased} if $\frac{1}{\sqrt{n}}HK^\top$ is a Hadamard matrix of order $n$.  
\begin{proposition}
Let $H=\begin{pmatrix}H_1\\H_2\end{pmatrix}$ be a balancedly splittable Hadamard matrix of order $n$ with $H_1^\top H_1=\ell I_n+a S$ where $a\neq 0$ and $S$ is an $n\times n$ $(0,1,-1)$-matrix. 
Then the following are equivalent.
\begin{enumerate}
\item $K:=\frac{1}{2a}(H_1^\top H_1-H_2^\top H_2)$ is a Hadamard matrix.
\item $(\ell,a)=((n\pm\sqrt{n})/2,\sqrt{n}/2)$. 
\end{enumerate}
In this case, $n=4k^2$ for some integer $k$ and the Hadamard matrices $H$ and $K$ are unbiased.  
\end{proposition}
\begin{proof}
Since $H_1 H_2^\top$ and $H_2^\top H_1^\top$ are zero matrices and $H_1H_1^\top=n I_{\ell}, H_2H_2^\top=n I_{n-\ell}$,
\begin{align*}
K K^\top&=K^2=\frac{1}{4a^2}(H_1^\top H_1-H_2^\top H_2)^2=\frac{1}{4a^2}(H_1^\top (H_1H_1^\top) H_1+H_2^\top (H_2H_2^\top) H_2)\\
&=\frac{n}{4a^2}( H_1^\top H_1+H_2^\top H_2)=\frac{n^2}{4a^2}I_n.
\end{align*}
Therefore $K$ is a Hadamard matrix if and only if $K$ is a $(1,-1)$-matrix and $a=\sqrt{n}/2$. 
Since $K=\frac{1}{2a}(H_1^\top H_1-H_2^\top H_2)=\frac{1}{2a}((2\ell-n)I_n-2a S)$, 
$K$ is a $(1,-1)$-matrix if and only if $(2\ell-n)/(2a)=\pm1$.
By Proposition~\ref{prop:bshp}(1) the latter is equivalent to $\ell=(n\pm\sqrt{n})/2$. 
Therefore (1) is equivalent to (2). 
  
If (1) and (2) hold, then $a=\sqrt{n}/2$ is an integer. 
Therefore $n$ must be a square of an even integer. 
And we have that 
$HK^\top=\sqrt{n}\begin{pmatrix}H_1 \\ -H_2 \end{pmatrix}$. 
Thus $H$ and $K$ are unbiased. 
\end{proof}

A Hadamard matrix of order $n$ is said to be \defn{regular} if ${\bf 1}^\top H=\sqrt{n}{\bf 1}^\top$. 
In this case $n$ must be square. 
\begin{proposition}\label{prop:regularHM}
Any balanced splittable Hadamard matrix of order $4n^2$ with the parameters $(\ell,a,b)=(2n^2-n,n,-n)$ is equivalent to a regular Hadamard matrix.  
\end{proposition}
\begin{proof}
Let $H=\begin{pmatrix} H_1\\H_2\end{pmatrix}$ where $H_1$ is an $\ell\times n$ matrix. 
Since multiplying a signed permutation matrix by $H$ from the left keeps the property of balancedly splittable, we may assume that $H_1$ has the all-ones first column and $H_2$ has the negative all-ones first column.    
By the assumption $b=-a$, multiplying a signed permutation matrix by $H$ from the right also keeps the property of balancedly splittable. 
This implies that ${\bf 1}^\top H_1=(-2n^2+n,n,\ldots,n)$ and  ${\bf 1}^\top H_2=(2n^2+n,n,\ldots,n)$.  
Therefore ${\bf 1}^\top H={\bf 1}^\top H_1+{\bf 1}^\top H_2=(2n,2n,\ldots,2n)$, which proves that $H$ is equivalent to a regular Hadamard matrix. 
\end{proof}
\begin{remark}
The Hadamard matrices of order $16$ with Hall's classes IV or V  are not equivalent to regular Hadamard matrices, and thus are not balancedly splittable with the parameters $(\ell,a,-b)=(6,2,-2)$ \cite{hkl}. See \cite{W77} for Hall's classes of Hadamard matrices.  
\end{remark}

We now present two non-existence results. 
\begin{proposition}\label{prop:non-exist1}
There is no balancedly splittable Hadamard matrix with the parameters $(n,\ell,a,-a)$, $\ell+a\not\equiv 0\pmod{4}$, $1<\ell<n-1$. 
\end{proposition}
\begin{proof}
Assume that there exists such a balancedly Hadamard matrix $H=\begin{pmatrix} H_1\\H_2\end{pmatrix}$ where $H_1$ is an $\ell\times n$ matrix. 
By  multiplying $H$ on both sides by signed permutation matrices, we may assume that $H_1^\top H_1$ has $a$ as its entries in the first row and the first column except $(1,1)$-entry.  
Now we claim that $H_1^\top H_1=\ell I+a(J-I)$. 
Indeed, suppose for the contrary that there exist two columns, say $i$-th and $j$-th columns, distinct from the first column such that their inner product equals to $-a$.  
Let $x,y,x,w$ be non-negative integers such that 
\begin{equation*}
\begin{array}{lccc}
\text{the first column}=(+ \cdots + & + \cdots + & + \cdots + & + \cdots + )^\top,\\
\text{the $i$-th column}=(+ \cdots + & + \cdots + & - \cdots - & - \cdots - )^\top,\\
\text{the $j$-th column}=(
\underbrace{+ \cdots +}_{x \text{ rows}} &
\underbrace{- \cdots -}_{y \text{ rows}} &
\underbrace{+ \cdots +}_{z \text{ rows}} &
\underbrace{- \cdots -}_{w \text{ rows}})^\top.
\end{array}
\end{equation*}
Then it follows that 
\begin{align*}
\begin{cases}
x+y+z+w&=\ell,\\
x+y-z-w&=a,\\
x-y+z-w&=a,\\
x-y-z+w&=-a.
\end{cases}
\end{align*}
Solving these equations yields $(x,y,z,w)=(\frac{\ell+a}{4},\frac{\ell+a}{4},\frac{\ell+a}{4},\frac{\ell-3a}{4})$, which is impossible because $\ell+a\not\equiv0\pmod{4}$. 
Therefore we have $H_1^\top H_1=\ell I+a(J-I)$. 

However, $H_1^\top H_1=\ell I+a(J-I)$ contradicts Proposition~\ref{prop:bsh1} by $1<\ell<n-1$. 
Therefore we conclude that such a balancedly splittable Hadamard matrix does not exist. 
\end{proof}

In a similar way the following is proved. 
\begin{proposition}\label{prop:non-exist2}
There is no balancedly splittable Hadamard matrix with the parameters $(n,\ell,a,-a)$, $\ell\not\equiv a\pmod{4}$ and $a>1$. 
\end{proposition}
\begin{proof}
In the same way as in Proposition~\ref{prop:non-exist1} we may assume that there exists such a balancedly Hadamard matrix $H=\begin{pmatrix} H_1\\H_2\end{pmatrix}$ where $H_1$ is an $\ell\times n$ matrix and $H_1^\top H_1$ has $a$ as its entries in the first row and the first column except $(1,1)$-entry.  
Now we claim that 
\begin{align*}
H_1^\top H_1=\begin{pmatrix}
\ell & a {\bf 1}^\top \\
a {\bf 1} & \ell I_{n-1} -a(J_{n-1}-I_{n-1}) \\
\end{pmatrix}. 
\end{align*}
Indeed, suppose to the contrary that there exist two columns, say $i$-th and $j$-th columns, distinct from the first column such that their inner product equals to $a$.  
Let $x,y,x,w$ be non-negative integers such that 
\begin{equation*}
\begin{array}{lccc}
\text{the first column}=(+ \cdots + & + \cdots + & + \cdots + & + \cdots + )^\top,\\
\text{the $i$-th column}=(+ \cdots + & + \cdots + & - \cdots - & - \cdots - )^\top,\\
\text{the $j$-th column}=(
\underbrace{+ \cdots +}_{x \text{ rows}} &
\underbrace{- \cdots -}_{y \text{ rows}} &
\underbrace{+ \cdots +}_{z \text{ rows}} &
\underbrace{- \cdots -}_{w \text{ rows}})^\top.
\end{array}
\end{equation*}
Then It is seen  that 
\begin{align*}
\begin{cases}
x+y+z+w&=\ell,\\
x+y-z-w&=a,\\
x-y+z-w&=a,\\
x-y-z+w&=a.
\end{cases}
\end{align*}
Solving these equations yields $(x,y,z,w)=(\frac{\ell+3a}{4},\frac{\ell-a}{4},\frac{\ell-a}{4},\frac{\ell-a}{4})$, which is impossible because $l\not\equiv a\pmod{4}$. 
Therefore we have $H_1^\top H_1=\begin{pmatrix}
\ell & a {\bf 1}^\top \\
a {\bf 1} & \ell I_{n-1} -a(J_{n-1}-I_{n-1}) \\
\end{pmatrix}
$. 
It can beseen that $\ell-(n-1)a$ is one of  the eigenvalues of $\begin{pmatrix}
\ell & a {\bf 1}^\top \\
a {\bf 1} & \ell I_{n-1} -a(J_{n-1}-I_{n-1}) \\
\end{pmatrix}$ 
However, $\ell-(n-1)a<0$ for  $a>1$, which contradicts the fact that all the singular values of $H_1$ are nonnegative.
Therefore we conclude that such a balancedly splittable Hadamard matrix does not exist. 
\end{proof}

There are exactly three inequivalent Hadamard matrices of order $16$ with maximal excess $64$. 
Those are contained in the Hall's classes I, II, and III.  The corresponding strongly regular graphs are $K_4 \times K_4$ and the Shrikhande graph \cite{Sh}). 
These three Hadamard matrices are balancedly splittable.  
The following are three examples of order $16$. 
\begin{example}\label{ex:syl}
The Hadamard matrix of order $16$ of Hall's class I, that is, the Sylvester Hadamard matrix is balancedely splittable with parameters $(16,6,2,-2)$. 
The corresponding strongly regular graph is $K_4\times K_4$.  
\end{example}

\begin{example}\label{ex:shr}
The Hadamard matrix of order $16$ of Hall's class II is balancedely splittable with parameters $(16,6,2,-2)$. 
The corresponding strongly regular graph is $K_4\times K_4$.   

\end{example}

\begin{example}
The Hadamard matrix of order $16$ of Hall's class III is balancedely splittable with parameters $(16,6,2,-2)$. 
The corresponding strongly regular graph is the Shrikhande graph \cite{Sh}.   

\end{example}

Though the classification of Hadamard matrices of order $36$ has not been finished yet, we have   the non-existence results for balancedly splittable Hadamard matrices of order $36$ by dealing with the eigenspaces of the attached strongly regular graphs.
\begin{proposition}\label{prop:ne36}
\begin{enumerate}
\item There is no balancedely splittable Hadamard matrix of order $36$ with the parameters $(36,10,4,-2)$. 
\item There is no balancedely splittable Hadamard matrix of order $36$ with the parameters $(36,25,1,-5)$.
\item There is no balancedely splittable Hadamard matrix of order $36$ with the parameters $(36,14,2,-4)$.
\item There is no balancedely splittable Hadamard matrix of order $36$ with the parameters $(36,20,2,-4)$.
\end{enumerate}
\end{proposition} 
\begin{proof}
(1): If there would exist such a Hadamard matrix $H$, then $H$ must come from the unique strongly regular graph with the parameters $(36,10,4,2)$ having the adjacency matrix $A$. 
The matrix $B:=12I+6A-2J$ is written as $12I+6A-2J=H_1^\top H_1$ where $H_1$ is a $10\times 36$ $(1,-1)$-matrix.  
Then the eigenvectors of $B$ with eigenvalue $36$ are the row vectors of the matrix $(I_{10},X)$ where  $X$ is
\begin{align*}
&\left(
\begin{smallmatrix}
 4 & - & - & - & - & 3 & - & - & - & - & 3 & - & - & - & - & 3 & - & - & - & - & 3 & - & - & - & - & 3 \\
 - & 1 & 1 & 1 & 1 & 0 & 0 & 0 & 0 & 0 & - & 0 & 0 & 0 & 0 & - & 0 & 0 & 0 & 0 & - & 0 & 0 & 0 & 0 & - \\
 - & 0 & 0 & 0 & 0 & - & 1 & 1 & 1 & 1 & 0 & 0 & 0 & 0 & 0 & - & 0 & 0 & 0 & 0 & - & 0 & 0 & 0 & 0 & - \\
 - & 0 & 0 & 0 & 0 & - & 0 & 0 & 0 & 0 & - & 1 & 1 & 1 & 1 & 0 & 0 & 0 & 0 & 0 & - & 0 & 0 & 0 & 0 & - \\
 - & 0 & 0 & 0 & 0 & - & 0 & 0 & 0 & 0 & - & 0 & 0 & 0 & 0 & - & 1 & 1 & 1 & 1 & 0 & 0 & 0 & 0 & 0 & - \\
 - & 0 & 0 & 0 & 0 & - & 0 & 0 & 0 & 0 & - & 0 & 0 & 0 & 0 & - & 0 & 0 & 0 & 0 & - & 1 & 1 & 1 & 1 & 0 \\
 - & 1 & 0 & 0 & 0 & - & 1 & 0 & 0 & 0 & - & 1 & 0 & 0 & 0 & - & 1 & 0 & 0 & 0 & - & 1 & 0 & 0 & 0 & - \\
 - & 0 & 1 & 0 & 0 & - & 0 & 1 & 0 & 0 & - & 0 & 1 & 0 & 0 & - & 0 & 1 & 0 & 0 & - & 0 & 1 & 0 & 0 & - \\
 - & 0 & 0 & 1 & 0 & - & 0 & 0 & 1 & 0 & - & 0 & 0 & 1 & 0 & - & 0 & 0 & 1 & 0 & - & 0 & 0 & 1 & 0 & - \\
 - & 0 & 0 & 0 & 1 & - & 0 & 0 & 0 & 1 & - & 0 & 0 & 0 & 1 & - & 0 & 0 & 0 & 1 & - & 0 & 0 & 0 & 1 & - 
\end{smallmatrix}
\right),
\end{align*}
and $-$ stands for $-1$. 
By computer search, there are no mutually orthogonal $10$ eigenvectors of $B$ with eigenvalue $36$ and entries $1,-1$. Therefore there is no Hadamard matrix $H$ of order $36$ such that any $10\times 36$ submatrix $H_1$ of $H$ satisfies that $12I+6A-2J=H_1^\top H_1$. 

The proofs for (2), (3), and (4) are the same as that of (1). 
\end{proof}
Note that the strongly regular graph for (2) is the complement of that for (1). There exist $180$ strongly regular graphs with the parameters $(36,21,12,12)$ which correspond to the case (3), and  there exist $32548$ strongly regular graphs with the parameters $(36,20,10,12)$ which correspond to the case (4).  
\section{Constructions}
In this section, 
we construct several balancedly splittable Hadamard matrices. 

\subsection{Construction for $(n,\ell,a,b)=(m^2,(m-1)^2,1,-m+1),(m^2,2m-2,m-2,-2)$, $m$ an order for a Hadamard matrix}
\begin{theorem}\label{thm:31}
Let $m$ be the order for a Hadamard matrix. 
Then there exists a balancedly splittable Hadamard matrix of order $m^2$ with the parameters $(m^2,(m-1)^2,1,-m+1)$ and $(m^2,2m-2,m-2,-2)$.   
\end{theorem}
\begin{proof}
Let $H$ be a Hadamard matrix of order $m$. 
Normalize $H$ so that $H=\begin{pmatrix}{\bf 1}^\top \\ H_1  \end{pmatrix}$.  
Then $H\otimes H$ is a Hadamard matrix and has $H_1\otimes H_1$ as a submatrix of $H\otimes H$. 
Then using the property that $H_1^\top H_1=m I_m-J_m$, we have 
\begin{align*}
(H_1\otimes H_1)^\top (H_1\otimes H_1)&=H_1^\top H_1\otimes H_1^\top H_1=(m I_m-J_m)\otimes (m I_m-J_m),
\end{align*}
which has only two distinct entries off diagonal. 
Therefore $H\otimes H$ is a balancedly splittable Hadamard matrix of order $m^2$ with the parameters $(m^2,(m-1)^2,1,-m+1)$. 
Note that $H\otimes H$ is normalized and the all-ones row vector is not a row vector of $H_1\otimes H_1$. 
Then we use the fact in Remark~\ref{rem:sreg} (2) to show that  $H\otimes H$ is also a balancedly splittable Hadamard matrix of order $m^2$ with the parameters $(m^2,2m-2,m-2,-2)$. 
\end{proof}

\subsection{Construction for $(n,\ell,a,b)=(m^2,m,m,0)$, $m$ an order for a Hadamard matrix}
\begin{theorem}
Let $m$ be the order for a Hadamard matrix. 
Then there exists a balancedly splittable Hadamard matrix of order $m^2$ with the parameters $(m^2,m,m,0)$.   
\end{theorem}
\begin{proof}
Let $H$ be a Hadamard matrix of order $m$. 
Let $r_i$ be the $i$-th row of $H$ for $i\in\{1,\ldots,m\}$ and normalize $H$ so that $r_1$ is the all-ones vector. 
Define an $m^2\times m^2$ matrix $M$ by $M=(r_j^\top r_i)_{i,j=1}^m$.   
Then $M$ is a Hadamard matrix of order $m^2$. 
Let $M_1=(r_j^\top  r_1)_{j=1}^m$ be a submatrix of $M$. By $r_i r_j^\top=m\delta_{i,j}$ and $r_1={\bf 1}_m^\top$, we have
\begin{align*}
M_1^\top M_1&=\begin{pmatrix}r_1^\top r_1\\ r_1^\top r_2\\ \vdots \\r_1^\top r_m\end{pmatrix}(r_1^\top r_1,r_2^\top r_1,\ldots,r_m^\top r_1)
=\begin{pmatrix}
m J_m & O & \cdots& O\\ 
O & m J_m  & \cdots & O\\
\vdots & \vdots  & \ddots & \vdots\\
O & O  & \cdots & m J_m\\
\end{pmatrix},
\end{align*}
which has only two distinct entries off diagonal. 
Therefore $M$ is a balancedly splittable Hadamard matrix of order $m^2$ with the parameters $(m^2,m,m,0)$.  
\end{proof}

\subsection{Construction for $(n,\ell,a,b)=(km,k(m-1),0,-k)$, $k,m$ orders for Hadamard matrices}
\begin{theorem}
Let $k,m$ be the orders for Hadamard matrices. 
Then there exists a balancedly splittable Hadamard matrix of order $km$ with the parameters $(km,k(m-1),0,-k)$.   
\end{theorem}
\begin{proof}
Let $H,K$ be Hadamard matrices of order $k,m$ respectively. 
Normalize $K$ so that $K=\begin{pmatrix}{\bf 1}^\top \\ K_1  \end{pmatrix}$. 
Then $H\otimes K$ is a Hadamard matrix and has $H\otimes K_1$ as a submatrix. 
Then using the property that $K_1^\top K_1=m I_m-J_m$, we have 
\begin{align*}
(H\otimes K_1)^\top (H\otimes K_1)&=H^\top H\otimes K_1^\top K_1=k I_k\otimes (m I_m-J_m),
\end{align*}
which has only two distinct entries off diagonal. 
Therefore $H\otimes K$ is a balancedly splittable Hadamard matrix of order $km$ with the parameters $(km,k(m-1),0,-k)$.  
\end{proof}

\subsection{Construction for $(n,\ell,a,b)=(n,n-2,0,-2)$, $n$ an order for a Hadamard matrix}
The following result is the same as \cite[Observation~2]{BFK}. 
\begin{theorem}
Let $n$ be the order for a Hadamard matrix. 
Then there exists a balancedly splittable Hadamard matrix of order $n$ with the parameters $(n,n-2,0,-2)$.   
\end{theorem}
\begin{proof}
Let $H$ be a Hadamard matrix of order $n$. 
Normalize the first two rows of $H$ so that $H=\left(
\begin{array}{c}
\begin{matrix}
{\bf 1}_{n/2}^\top & {\bf 1}_{n/2}^\top  \\
{\bf 1}_{n/2}^\top & -{\bf 1}_{n/2}^\top
\end{matrix} \\
H_1\end{array}
\right)$. 
Then  
$H_1^\top H_1=\begin{pmatrix}nI_{n/2}-2J_{n/2} & O_{n/2}\\O_{n/2} & nI_{n/2}-2J_{n/2} \end{pmatrix}$,
which has only two distinct entries off diagonal, where $O_{n/2}$ denotes the zero matrix of order $n/2$. 
It follows that $H$ is a balancedly splittable Hadamard matrix of order $n^2$ with the parameters $(n,n-2,0,-2)$.  
\end{proof}


\subsection{Construction for $(n,\ell,a,b)=(4^m,2^m,2^m,0),(4^m,2^{m-1}(2^m-1),2^{m-1},-2^{m-1})$, $m$ a positive integer}\label{subsec:hamming}
Let $H_1=\begin{pmatrix}1 & 1 \\ 1 & -1\end{pmatrix}$ and define $H_m=H_{m-1} \otimes H_1$ recursively for $m\geq2$. 
Then $H_m$ is a Hadamard matrix of order $2^m$, which is called \defn{Sylvester-type}.

\begin{lemma}\label{lem:recursive0}
If there exist a balancedly splittable Hadamard matrix of order $n_i^2$ with $(\ell_i,a_i,b_i)=(n_i,n_i,0)$ for $i=1,2$, 
then there exists a balancedly splittable Hadamard matrix $H$ of order $n_1^2n_2^2$ with $(\ell,a,b)=(n_1n_2,n_1n_2,0)$.   
\end{lemma}
\begin{proof}
Let $H_i=\begin{pmatrix}H_{i,1}\\H_{i,2}\end{pmatrix}$ be balancedly splittable Hadamard matrices of order $n_i^2$ with $(\ell_i,a_i,b_i)=(n_i,n_i,0)$ with respect to $H_{i,1}$ for $i=1,2$. 
Then, by Remark~\ref{rem:imp}, $H_{i,1}^\top H_{i,1}=n_iI_{n_i}\otimes J_{n_i}$ for $i=1,2$. 
A Hadamard matrix $H_1\otimes H_2$ has a submatrix $H_{1,1}\otimes H_{2,1}$. Then, 
\begin{align*}
(H_{1,1}\otimes H_{2,1})^\top (H_{1,1}\otimes H_{2,1})&=H_{1,1}^\top H_{1,1}\otimes H_{2,1}^\top H_{2,1}=n_1n_2 I_{n_1}\otimes J_{n_1}\otimes I_{n_2}\otimes J_{n_2},  
\end{align*}
which is permutationally equal to $n_1n_2 I_{n_1n_2}\otimes J_{n_1n_2}$.  
This proves that $H_1\otimes H_2$ is balancedly splittable. 
\end{proof}

\begin{lemma}\label{lem:recursive}
If there exist a balancedly splittable Hadamard matrix of order $n_i$ with $(\ell_i,a_i,b_i)=((n_i+\sqrt{n_i})/2,\sqrt{n_i}/2,-\sqrt{n_i}/2)$ for $i=1,2$, 
then there exists a balancedly splittable Hadamard matrix $H$ of order $n_1n_2$ with $(\ell,a,b)=((n_1n_2+\sqrt{n_1n_2})/2,\sqrt{n_1n_2}/2,-\sqrt{n_1n_2}/2)$.   
\end{lemma}
\begin{proof}
Let $H_i=\begin{pmatrix}H_{i,1}\\H_{i,2}\end{pmatrix}$ be balancedly splittable Hadamard matrices of order $n_i$ with $(\ell_i,a_i,b_i)=((n_i+\sqrt{n_i})/2,\sqrt{n_i}/2,-\sqrt{n_i}/2)$ for $i=1,2$. 
Then 
$H_{i,1}^\top H_{i,1}=\ell_i I_{n_i}+a_i S_i$
for some Seidel matrices $S_i$, $i=1,2$. 
We consider a Hadamard matrix $H_1\otimes H_2$, and  it has $K:=\begin{pmatrix} H_{1,1}\otimes H_{2,1} \\ H_{1,2}\otimes H_{2,2}\end{pmatrix}$ as a submatrix.
Then 
\begin{align*}
&K^\top K=(H_{1,1}\otimes H_{2,1})^\top (H_{1,1}\otimes H_{2,1})+(H_{1,2}\otimes H_{2,2})^\top (H_{1,2}\otimes H_{2,2})\\
&=H_{1,1}^\top H_{1,1}\otimes H_{2,1}^\top H_{2,1}+H_{1,2}^\top H_{1,2}\otimes H_{2,2}^\top H_{2,2}\\
&=(\ell_1 I_{n_1}+a_1 S_1) \otimes (\ell_2 I_{n_2}+a_2 S_2)+((n_1-\ell_1) I_{n_1}-a_1 S_1) \otimes ((n_2-\ell_2)-a_2 S_2)\\
&=(\ell_1\ell_2+(n_1-\ell_1)(n_2-\ell_2))I_{n_1n_2}+(2\ell_1-n_1)a_2I_{n_1}\otimes S_2+a_1(2\ell_2-n_2)S_1\otimes I_{n_2}+2a_1a_2S_1\otimes S_2 \\
&=\frac{n_1n_2+\sqrt{n_1n_2}}{2}I_{n_1n_2}+\frac{\sqrt{n_1n_2}}{2}I_{n_1}\otimes S_2+\frac{\sqrt{n_1n_2}}{2}S_1\otimes I_{n_2}+\frac{\sqrt{n_1n_2}}{2} S_1\otimes S_2, 
\end{align*}
which has only two distinct off-diagonal entires. 
Thus $H_1\otimes H_2$ is balancedly splittable. 
\end{proof}


A balancedly splittable Hadamard matrix $H$ of order $n^2$ is said to be \defn{twin} if $H=\begin{pmatrix}H_{1}\\H_{2}\\H_{3}\end{pmatrix}$  satisfies that $H$ is balancedly splittable with parameters $(n^2,n,n,0)$ with respect to $H_1$ and with parameters $(n^2,n(n-1)/2,n/2,-n/2)$ with respect to $H_2$ and $H_3$. 
\begin{theorem}
The Sylvester-type Hadamard matrix of order $4^m$ is twin balancedly splittable.
\end{theorem}
\begin{proof}
Let $H$ be the Sylvester-type Hadamard matrix of order $4$:
\begin{align*}
H=\begin{pmatrix}H_{1}\\H_{2}\\H_{3}\end{pmatrix}=\begin{pmatrix}
1 & 1 & 1 & 1 \\
1 & 1 & -1 & -1 \\
1 & -1 & 1 & -1 \\
1 & -1 & -1 & 1 
\end{pmatrix},
\end{align*}
where $H_1$ is a $2\times 4$ matrix and $H_2,H_3$ are both $1\times 4$ matrices. 
The result follows from Lemma~\ref{lem:recursive0}, Lemma~\ref{lem:recursive} and the above Hadamard matrix of order $4$. 
\end{proof}

\subsection{Construction for $(n,\ell,a,b)=(q(q+1),q,q,-1)$, where $q$ an order of a skew-symmetric Hadamard matrix}\label{sec:constq(q+1)}
In \cite{Wi}, it is shown that the existence of a skew-symmetric Hadamard matrix of order $q+1$ implies that  the existence of a Hadamard matrix of order $(q-1)q$. 
We review the construction and its generalization. 

In \cite{FKS}, the following matrices $\mathcal{J}_m^{(q)}$ and $\mathcal{A}_m^{(q)}$  are used in order to construct a quaternary complex Hadamard matrix.
Let  $q+1$ be the order of a skew type Hadamard matrix $H$. Multiply some rows and columns of $H$ by $-1$, if necessary, we may assume that \[ H=\left (\begin{matrix} 1 & {\bf 1}^\top\\-{\bf 1}  & I+Q\end{matrix}\right).\] 
The $(0,\pm 1)$-matrix $Q=(q_{ij})_{i,j=1}^{q}$, called the \emph{skew symmetric core} of the skew type Hadamard matrix, is
skew symmetric and satisfies that $J_qQ=QJ_q=O_q$, and $QQ^\top =qI_q-J_q$. 
For any odd prime power $q$, see \cite{P} for Paley's construction.  

Let $q$ be the order of a skew symmetric core $Q$. Define the following matrices recursively for each nonnegative integer $m$:
\begin{align}
\label{eqn:JA}
\Jq_m & = \begin{cases} 
J_1 & \text{if } m=0,\\
J_q \otimes \mathcal{A}^{(q)}_{m-1} & \text{otherwise},
\end{cases}    &    
\Aq_m & = \begin{cases} 
J_1 & \text{if } m=0,\\
I_q \otimes \mathcal{J}^{(q)}_{m-1} + Q \otimes \mathcal{A}^{(q)}_{m-1} & \text{otherwise.}
\end{cases}   
\end{align}
For a normalized Hadamard matrix of $H$ of order $q+1$ with skew symmetric core $Q$, the matrix $C=H-I$ is a \emph{conference matrix}, that is, $CC^\top =q I$. 
We define $M=-I_{q+1}\otimes \Jq_1+C\otimes \Aq_1$. 
 \begin{theorem}
The matrix $M$ is a balancedly splittable Hadamard matrix of order $q(q+1)$ with $(n,\ell,a,b)=(q(q+1),q,q,-1)$. 
\end{theorem}
\begin{proof}
To use the properties that $\Jq_1(\Jq_1)^\top +q\Aq_1(\Aq_1)^\top =q(q+1)I_q$ and $\Jq_1(\Aq_1)^\top =\Aq_1(\Jq_1)^\top $,  we have 
\begin{align*}
MM^\top &=(-I_{q+1}\otimes \Jq_1+C\otimes \Aq_1)(-I_{q+1}\otimes (\Jq_1)^\top +C^\top \otimes (\Aq_1)^\top )\\
&=I_{q+1}\otimes \Jq_1 (\Jq_1)^\top -C\otimes \Aq_1 (\Jq_1)^\top -C^\top \otimes  \Jq_1(\Aq_1)^\top +CC^\top \otimes  \Aq_1(\Aq_1)^\top \\
&=I_{q+1}\otimes (\Jq_1 (\Jq_1)^\top +q \Aq_1(\Aq_1)^\top )-C\otimes (\Aq_1 (\Jq_1)^\top -\Jq_1(\Aq_1)^\top )\\
&=q(q+1)I_{q+1}\otimes I_q.
\end{align*}
Therefore $M$ is a Hadamard matrix.  Next we show that $M$ is balancedly splittable with respect to $M_1$ obtained from $M$ by restricting  rows to the first $q$ rows. 

Since $M_1=\begin{pmatrix}-\Jq_1 & \Aq_1 &\cdots & \Aq_1 \end{pmatrix}$ and 
\begin{align*}
(\Jq_1)^\top \Jq_1&=(J_q)^2=qJ_q,\\
(\Jq_1)^\top \Aq_1&=J_q(I_q+Q)=J_q,\\
(\Aq_1)^\top \Aq_1&=(I_q+Q^\top )(I_q+Q)=I_q+Q^\top Q=(q+1)I_q-J_q,
\end{align*}
we have 
\begin{align*}
M_1^\top  M_1
&=\begin{pmatrix} 
q J_q & -J_q & \cdots & -J_q \\
-J_q & (q+1)I_q-J_q & \cdots & (q+1)I_q-J_q \\
\vdots & \vdots & \ddots & \vdots \\
-J_q & (q+1)I_q-J_q & \cdots & (q+1)I_q-J_q 
\end{pmatrix}. 
\end{align*}
Thus $M$ is balancedly splittable.     
\end{proof}

\section{Commutative association schemes}\label{sec:as} 
In this section we define commutative association schemes. 

A \emph{$d$-class commutative association scheme}, see \cite{BI},
with a finite vertex set $X$,  
is a set of non-zero $(0,1)$-matrices $A_0, A_1,\ldots, A_d$ with
rows and columns indexed by $X$, such that
\begin{enumerate}
\item $A_0=I_{|X|}$,
\item $\sum_{i=0}^d A_i = J_{|X|}$,
\item $A_i^\top\in\{A_1,\ldots,A_d\}$ for any $i\in\{1,\ldots,d\}$, 
\item for all $i$, $j$, $A_iA_j=\sum_{k=0}^d p_{i,j}^k A_k$ for some non-negative integers 
$p_{i,j}^k$,
\item for all $i$, $j$, $A_iA_j=A_jA_i$.
\end{enumerate}
The association scheme is said to be \defn{symmetric} if $A_i^\top=A_i$ for any $i$, and \defn{non-symmetric} otherwise.  
Note that if symmetric matrices $A_i$ ($0\leq i\leq d$) satisfy (4), then (5) must follow. 
The vector space spanned by $A_i$'s over the real number field forms a commutative algebra, denoted by $\mathcal{A}$ and is called the \emph{Bose-Mesner algebra}.
Then there exists a basis of $\mathcal{A}$ consisting of primitive idempotents, say $E_0=(1/|X|)J_{|X|},E_1,\ldots,E_d$. 
Since  $\{A_0,A_1,\ldots,A_d\}$ and $\{E_0,E_1,\ldots,E_d\}$ are two bases of $\mathcal{A}$, there exist the change-of-basis matrices $P=(P_{i,j})_{i,j=0}^d$, $Q=(Q_{i,j})_{i,j=0}^d$ so that
\begin{align*}
A_i=\sum_{j=0}^d P_{j,i}E_j,\quad E_j =\frac{1}{|X|}\sum_{i=0}^{d} Q_{i,j}A_i.
\end{align*}   
The matrices $P,Q$ are said to be the \defn{first and second eigenmatrices} respectively.  

\begin{example}
Construction in Subsection~\ref{subsec:hamming} is closely related to the binary Hamming schemes $H(n,2)$.   
Let $X=\mathbb{Z}_2^n$ and $R_i=\{(x,y)\mid x,y\in X, d(x,y)=i\}$ 
for $i=0,1,\ldots,n$, where $d(x,y)$ is the Hamming distance between $x$ and $y$. 
Define $A_i$ to be the adjacency matrix of a graph $(X,R_i)$ for $i=0,1,\ldots,n$. 
Then the matrices $A_0,A_1,\ldots,A_n$ is a symmetric association scheme,
which is called the {\em binary Hamming} association scheme, denoted by $H(n,2)$.

We denote adjacency matrices of the Hamming scheme $H(n,2)$ by $A_i^{(n)}$.

The Hamming scheme $H(n+1,2)$ is described as a fusion scheme of the product of schemes $H(n,2)$ and $H(1,2)$ as follows, see also \cite{BHOS}.
For association schemes $\{A_0',A_1',\ldots,A_{d_1}'\}$ and $\{A_0'',A_1'',\ldots,A_{d_2}''\}$, 
the product of these two is an association schemes with non-zero matrices $(0,1)$ 
$A_i'\otimes A_j''$ where $0\leq i\leq d_1, 0\leq j\leq d_2$. 
Take two association schemes as the Hamming schemes $H(n,2)$ and $H(1,2)$ respectively, 
then we have 
\[
A_i^{(n+1)}=\sum_{j+k=i}A_j^{(n)}\otimes A_k^{(1)}=A_i^{(n)}\otimes A_0^{(1)}+A_{i-1}^{(n)}\otimes A_1^{(1)}
\] 
for $i\in\{1,\ldots,n+1\}$. 
It follows now that the adjacency matrices of the binary Hamming scheme are diagonalizable by the Sylvester Hadamard matrices. 
For $n=1$, the adjacency matrices $A_0=I_2,A_1=J_2-I_2$ are diagonalizable by $H_1=\begin{pmatrix}1 & 1 \\ 1 & -1 \end{pmatrix}$. 
For $n\geq2$, it follows from the recurrence above that $H_n=H_1^{\otimes n}$ diagonalizes the adjacency matrices of $H(n,2)$. 

A \defn{subscheme} or \defn{fusion scheme} of the association scheme $(X,\{R_i\}_{i=0}^d)$ is an association scheme $(X,\{ \cup_{j\in \Lambda_i}R_j\}_{i=0}^{e} )$ for some decomposition $\{\Lambda_0,\Lambda_1,\ldots,\Lambda_e\}$ of $\{0,1,\ldots,d\}$ such that $\Lambda_0=\{0\}$. 

Muzychuk \cite{M} classified the subschemes of $H(n,2)$ for $n\geq 9$. By \cite[Theorem~2.1]{M}, there are exactly two cases for the subschemes to be primitive  strongly regular graphs:
\begin{itemize}
\item $n$ is even and $\Lambda_1=\{k\mid 1\leq k\leq n,k\equiv 0,1\pmod{4}\},\Lambda_2=\{k\mid 1\leq k\leq n,k\equiv 2,3\pmod{4}\}$,
\item $n$ is even and $\Lambda_1=\{k\mid 1\leq k\leq n,k\equiv 0,3\pmod{4}\},\Lambda_2=\{k\mid 1\leq k\leq n,k\equiv 1,2\pmod{4}\}$.
\end{itemize}
The parameters of these strongly regular graphs are 
\begin{align*}
(n,k,\lambda,\mu)=&(4^m,2^{m-1}(2^m\pm 1),2^{m-1}(2^{m-1}\pm 1),2^{m-1}(2^{m-1}\pm1))
\end{align*}
and their complements. 

The Doob schemes are the association schemes with the same parameters as the binary Hamming schemes \cite{D}. 
By Example~\ref{ex:shr}, the Doob schemes are Hadamard diagonalizable, and this scheme has the fusion schemes which yield strongly regular graphs.  
\end{example}

\section{Construction of commutative association schemes}
Let $H$ be a Hadamard matrix of order $n$ with rows $r_1,\ldots,r_n$.   
For $i\in\{1,\ldots,n\}$, let $C_i=r_i^\top r_i$. 
We call $C_1,\ldots,C_n$ the \defn{auxiliary matrices} of $H$. The auxiliary matrices play an important role to construct association schemes.  
The following is basic properties for the auxiliary matrices.  
\begin{lemma}\label{lem:haux}{\rm \cite{K}}
\begin{enumerate}
\item $\sum_{i=1}^{n}C_i=n I_n$. 
\item For any $i\in\{1,\ldots,n\}$, $C_i^2=n C_i$. 
\item For any distinct $i,j\in\{1,\ldots,n\}$, $C_iC_j=O_n$. 
\end{enumerate}
\end{lemma}
Note that for a Hadamard matrix $H$, letting $H=\begin{pmatrix} H_1\\H_2\end{pmatrix}$ where $H_1$ is an $\ell\times n$ matrix, it holds that $\sum_{i=1}^\ell C_i=H_1^\top H_1$. 

Some combinatorial objects and association schemes are obtained from a balancedly splittable Hadamard matrix of order $n$ such that $\sum_{i=1}^\ell C_i$ has exactly one off-diagonal entries and some Latin squares as follows:
\begin{itemize}
\item symmetric or skew-symmetric Bush type Hadamard matrices and $3$-class association schemes from the case $(\ell,a)=(n,0)$ with $C_1=J_n$ and a symmetric Latin squares with constant diagonal, as described in \cite{W}, \cite{GC}.
\item symmetric or skew-symmetric regular $(0,\frac{1}{n-1})$ biangular matrices and $4$-class association schemes from the case $(\ell,a)=(n-1,1)$ with $C_1=J_n$ and a symmetric Latin square  with constant diagonal, refer to \cite{KS}.
\item unbiased Hadamard matrices and $4$-class association schemes from the case $(\ell,a)=(n,0)$ and mutually unique fixed symbol (UFS) Latin squares, see \cite{HKO,LMO}.
\item unbiased Bush-type Hadamard matrices and $5$-class association schemes from the case $(\ell,a)=(n,0)$ with $C_1=J_n$ and mutually UFS Latin squares as defined in \cite{KSS}.
\item unbiased biangular vectors (more generally linked systems of symmetric group divisible designs) and $5$-class association schemes from the case $(\ell,a)=(n-1,1)$ with $C_1=J_n$ and mutually UFS Latin squares, as shown in \cite{HKS}, and \cite{KS2018}. 
\end{itemize}

In the following subsections, we construct symmetric or non-symmetric association schemes with $4$, $5$ or $6$-classes from a balancedly splittable Hadamard matrix such that $\sum_{i=1}^{\ell}C_i$ has exactly two distinct off diagonal entries and some Latin squares. 
Throughout the following subsections, we assume that  $H$ is a balancedly splittable Hadamard matrix of order $n$ with auxiliary matrices $C_1,\ldots,C_n$ satisfying $\sum_{i=1}^\ell C_i=\ell I_n+a A+b(J_n-A-I_n)$ where $A$ is an $n\times n$ $(0,1)$-matrix, $a\neq b$, and $C_iJ_n=O_n$ for $i\in\{1,\ldots,\ell\}$. 
According to Proposition~\ref{prop:bshp},  $b=\frac{\ell (-a+\ell-n)}{a (n-1)+\ell}$ and the matrix $A$ is the adjacency matrix of a strongly regular graph with the parameters $(n,k,\lambda,\mu)$ given as: 
\begin{align*}
k&=\frac{\ell n (n-\ell-1)}{n (a^2+\ell)-(a-\ell)^2},\\
\lambda&=\frac{n (n^2 (a^3+\ell^2)-2 (\ell+1) n (a^3+\ell^2)+(2 a \ell+a+\ell (\ell+2)) (a-\ell)^2)}{((a-\ell)^2-n (a^2+\ell))^2},\\
\mu&=\frac{\ell n (a-\ell) (\ell-n+1) (a-\ell+n)}{((a-\ell)^2-n (a^2+\ell))^2}. 
\end{align*}
We use $C_0:=O_n$ and a Latin square $L=({L_{i,j}})_{i,j\in S}$ on the symbol set $S$  where $S$ equals to $\{1,\ldots,\ell\}$ or $\{0,1,\ldots,\ell\}$, and denote $\tilde{L}$ to be
\begin{align*}
\tilde{L}=(C_{L_{i,j}})_{i,j\in S}. 
\end{align*} 

For the remaining part of the paper, we use a variant of {\it Mutually Orthogonal Latin Squares (MOLS)} which we call UFS  (Unique Fix Symbol) suitable for the way we apply it. Two Latin squares $L_1$ and $L_2$ of size $n$ on the same symbol set 
are called to be 
{\it UFS Latin squares},  if every superimposition of each row of $L_1$ on each row of $L_2$ results in
only one element of the form $(a,a)$. In effect,  each permutation of symbols between the rows of the two Latin squares has a Unique Fixed Symbol.
A set of Latin squares in which every distinct pair of Latin squares are UFS Latin square
is called
\emph{mutually UFS Latin squares}. 
Note that UFS Latin squares are called suitable Latin squares in \cite{HKO} and elsewhere. 
See \cite{HKO} for the equivalentness of existence between mutually UFS Latin squares and mutually orthogonal Latin squares. 
\begin{lemma}\label{lem:sl}
Let  $L_1,L_2$ be UFS Latin squares on the symbol set $\{1,\ldots,n\}$ with the $(i,j)$-entry equal to $l(i,j),l'(i,j)$ respectively.  
An $n\times n$ array with the $(i,j)$-entry equal to $b$ determined by $b=l(i,a)=l'(j,a)$ for the unique $a\in\{1,\ldots,n\}$, is a Latin square.
\end{lemma}

The following lemma will be used in Subsections~\ref{sec:as4}, \ref{sec:as5}, \ref{sec:as6}. 
We omit its easy proof. 

\begin{lemma}\label{lem:c}
Let $H$ be a balancedly splittable Hadamard matrix of order $n$. 
If $C_iJ_n=J_nC_i=O_n$, then $AC_i=C_iA=(n-\ell+b)C_i$.  
\end{lemma}

\begin{lemma}\label{lem:L}
Let $H$ be a balancedly splittable Hadamard matrix of order $n$ with $\sum_{i=1}^\ell C_i=\ell I_n+a A+b(J_n-A-I_n)$ and $C_iJ_n=O_n$ for $i\in\{1,\ldots,\ell\}$. 
Let $L$ be a Latin square on the symbol set $S$ where $S$ equals to $\{1,\ldots,\ell\}$ or $\{0,1,\ldots,\ell\}$.  
Then $\tilde{L}\tilde{L}^\top=n I_{|S|}\otimes (|S| I_n+a A+b(J_n-A-I_n))$. 
\end{lemma}

\begin{lemma}\label{lem:SL}
Let $H$ be a balancedly splittable Hadamard matrix of order $n$ with $\sum_{i=1}^\ell C_i=\ell I_n+a A+b(J_n-A-I_n)$ and $C_iJ_n=O_n$ for $i\in\{1,\ldots,\ell\}$.   
Let $L_1,\ldots,L_f$ be mutually UFS Latin squares on the symbol set $S$ where $S$ equals to $\{1,\ldots,\ell\}$ or $\{0,1,\ldots,\ell\}$.   
For distinct $i,j\in\{1,\ldots,f\}$, $\tilde{L}_i\tilde{L}_j^\top=n \tilde{L}_{i,j}$, where $L_{i,j}$ is the Latin square determined from $L_1,L_2$ by Lemma~\ref{lem:sl}.
\end{lemma}
Then the following holds: for any distinct $i,j,k\in\{1,\ldots,f\}$, 
$L_{i,k}$ and $L_{j,k}$ are UFS and the Latin square obtained from $L_{i,k}$ and $L_{j,k}$ in this ordering via Lemma~\ref{lem:sl} coincides with $L_{i,j}$.

\subsection{Symmetric and non-symmetric association schemes with $4$-classes}\label{sec:as4}
In this subsection, we will use a symmetric Latin square with constant diagonal, which is known to exist for order $v$ a positive even integer, see \cite{K}. 
Assume that $\ell$ is an odd integer and let $L$ be a symmetric Latin square of order $\ell+1$ on the symbol set $\{0,1,\ldots,\ell\}$ with constant diagonal $0$. 

We define disjoint $(0,1)$-matrices $A_i$ ($i\in\{0,1,\ldots,4\}$) as 
\begin{align*}
A_0=I_{(\ell+1)n},\quad A_1=I_{\ell+1}\otimes A,\quad A_2=I_{\ell+1}\otimes (J_n-A-I_n),\quad
\tilde{L}=A_3-A_4.
\end{align*}

\begin{theorem}\label{thm:as41}
The set of matrices $\{A_0,A_1,\ldots,A_4\}$ is a symmetric association scheme with $4$-classes. 
\end{theorem}
\begin{proof}
It is routine to see  that $A_0=I_{(\ell+1)n}$, $A_i$'s are disjoint symmetric $(0,1)$-matrices such that $\sum_{i=0}^4 A_i=J_{(\ell+1)n}$, and each $A_i$ is symmetric. 
Let $\mathcal{A}=\textrm{span}_{\mathbb{R}}\{A_0,A_1,\ldots,A_4\}$.  We will check 
(4) in the definition of the association scheme for each case. 

(i): For $i,j\in\{1,2\}$, (AS4) follows from the fact that $A$ is the adjacency matrix of a strongly regular graph. 

(ii): It follows from Lemma~\ref{lem:c} that $A_i(A_3-A_4),(A_3-A_4)A_i\in \mathcal{A}$ for $i=1,2$. 
Since $A_3+A_4=I_{\ell+1}\otimes(J_n-I_n)$ and $A$ is in particular the adjacency matrix of a regular graph, $A_i(A_3+A_4),(A_3+A_4)A_i\in \mathcal{A}$ for $i=1,2$. 
Thus (AS4) holds for $(i,j)\in(\{1,2\}\times\{3,4\})\cup(\{3,4\}\times\{1,2\})$.

(iii): For $i,j\in\{3,4\}$, $(A_3-A_4)^2=(\tilde{L})^2\in\mathcal{A}$ by Lemma~\ref{lem:L}. By $A_3+A_4=(J_{\ell+1}-I_{\ell+1})\otimes J_n$ and $C_iJ_n=O_n$ for $i\in\{1,\ldots,\ell\}$, $(A_3+A_4)(A_3-A_4)=(A_3-A_4)(A_3+A_4)=O_n\in\mathcal{A}$ and $(A_3+A_4)^2\in\mathcal{A}$. 
These implies that each component of $(A_3^2,A_3A_4,A_4A_3,A_4^2)H$ belongs to $\mathcal{A}$ where $H$ is a Hadamard matrix. 
Since $H$ is invertible, each of $A_3^2,A_3A_4,A_4A_3,A_4^2$ belongs to $\mathcal{A}$. 

This completes the proof. 
\end{proof}

Then the eigenmatrices $P$ and $Q$ are as follows:
\begin{align*}
P&=\left(
\begin{array}{ccccc}
  1 & \frac{\ell (n-\ell-1) n}{(n-1) a^2+2 \ell a+\ell (n-\ell)} & \frac{(\ell+a (n-1))^2}{(n-1) a^2+2 \ell a+\ell (n-\ell)} & \frac{\ell n}{2} & \frac{\ell n}{2} \\
 1 & \frac{\ell (n-\ell-1) n}{(n-1) a^2+2 \ell a+\ell (n-\ell)} & \frac{(\ell+a (n-1))^2}{(n-1) a^2+2 \ell a+\ell (n-\ell)} & -\frac{n}{2} & -\frac{n}{2} \\
 1 & \frac{a (n-\ell-1) n}{(n-1) a^2+2 \ell a+\ell (n-\ell)} & \frac{(\ell+a (n-1)) (n+a-\ell)}{(n-1) a^2+2 \ell a+\ell (n-\ell)} & -\frac{n}{2} & \frac{n}{2} \\
 1 & \frac{a (n-\ell-1) n}{(n-1) a^2+2 \ell a+\ell (n-\ell)} & \frac{(\ell+a (n-1)) (n+a-\ell)}{(n-1) a^2+2 \ell a+\ell (n-\ell)} & \frac{n}{2} & -\frac{n}{2} \\
 1 & -\frac{(a+1) \ell n}{(n-1) a^2+2 \ell a+\ell (n-\ell)} & -\frac{(a-\ell) (\ell+a (n-1))}{(n-1) a^2+2 \ell a+\ell (n-\ell)} & 0 & 0 
\end{array}
\right),\\
Q&=\left(
\begin{array}{ccccc}
1 & \ell & \frac{1}{2} \ell (\ell+1) & \frac{1}{2} \ell (\ell+1) & (\ell+1) (n-\ell-1) \\
 1 & \ell & \frac{1}{2} a (\ell+1) & \frac{1}{2} a (\ell+1) & -(a+1) (\ell+1) \\
 1 & \ell & \frac{\ell (\ell+1) (\ell-n-a)}{2 (\ell+a (n-1))} & \frac{\ell (\ell+1) (\ell-n-a)}{2 (\ell+a (n-1))} & \frac{(\ell-a) (\ell+1) (n-\ell-1)}{\ell+a (n-1)} \\
 1 & -1 & \frac{-\ell-1}{2} & \frac{\ell+1}{2} & 0 \\
 1 & -1 & \frac{\ell+1}{2} & \frac{-\ell-1}{2} & 0 \\
\end{array}
\right).
\end{align*}

By a slight modification, we obtain non-symmetric association schemes with $4$-classes.  
Under the same setting on $L$, $\ell$, and $C_i$ as above, 
we define $\bar{L}=(\epsilon_{i,j}C_{L_{i,j}})_{i,j=1}^{\ell+1}$, where 
$\epsilon_{i,j}=1$ if $i\leq j$ and $-1$ if $i>j$. 
We define disjoint $(0,1)$-matrices $A_i$ ($i\in\{0,1,\ldots,4\}$) as 
\begin{align*}
A_0=I_{(\ell+1)n},\quad A_1=I_{\ell+1}\otimes A,\quad A_2=I_{\ell+1}\otimes (J_n-A-I_n),\quad
\bar{L}=A_3-A_4.
\end{align*}
Note that $A_3^\top=A_4$. 
\begin{theorem}
The set of matrices $\{A_0,A_1,\ldots,A_4\}$ is a non-symmetric association scheme with $4$-classes. 
\end{theorem}
\begin{proof}
The proof is the same as that of Theorem~\ref{thm:as41}. 
\end{proof}
The eigenmatrices $\tilde{P}$ and $\tilde{Q}$ are obtained from $P,Q$ by changing $\tilde{P}_{i,j}=\sqrt{-1}P_{i,j}$ for $i\in\{2,3\},j\in\{3,4\}$ and $\tilde{Q}_{i,j}=\sqrt{-1}Q_{i,j}$ for $i\in\{3,4\},j\in\{2,3\}$.  

\subsection{Symmetric association schemes with $5$-classes}\label{sec:as5}
Let $L_1,\ldots,L_f$ be mutually UFS Latin squares on $\{1,\ldots,\ell\}$. 
We now construct a symmetric association scheme with $5$-classes from a balancedly splittable Hadamard matrix and mutually UFS Latin squares. 
Consider the Gram matrix $G$ of the row vectors of $\tilde{L}_i$ ($i\in\{1,\ldots,f\}$) defined by 
\begin{align*}
G=\begin{pmatrix}
\tilde{L}_1 \tilde{L}_1^\top & \tilde{L}_1 \tilde{L}_2^\top & \cdots & \tilde{L}_1 \tilde{L}_f^\top\\
\tilde{L}_2 \tilde{L}_1^\top & \tilde{L}_2 \tilde{L}_2^\top & \cdots & \tilde{L}_2 \tilde{L}_f^\top\\
\vdots & \vdots  & \ddots & \vdots\\
\tilde{L}_f \tilde{L}_1^\top & \tilde{L}_f \tilde{L}_2^\top & \cdots & \tilde{L}_f \tilde{L}_f^\top
\end{pmatrix}
\end{align*}

The entries of $G$ are $\{n\ell, n a, n b ,\pm n,0\}$. Decompose the matrix $G$ into its entries as 
\begin{align*}
G=n\ell A_0+n(aA_1+bA_2)+n(A_3-A_4). 
\end{align*}
Then the disjoint $(0,1)$-matrices $A_i$ ($i\in\{0,1,\ldots,4\}$) satisfy $\sum_{i=0}^4 A_i=J_{f \ell n}-I_f\otimes (J_{\ell}-I_{\ell})\otimes J_n$. 
We now define 
\begin{align*}
A_5=I_f\otimes (J_{\ell}-I_{\ell})\otimes J_n.
\end{align*}
Note that $A_1=I_f\otimes I_{\ell}\otimes A,
A_2=I_f\otimes I_{\ell}\otimes (J_n-A-I_n)$ and 
\begin{align*}
A_3-A_4&=\frac{1}{n}\begin{pmatrix}
O_{\ell n} & \tilde{L}_1 \tilde{L}_2^\top & \cdots & \tilde{L}_1 \tilde{L}_f^\top\\
\tilde{L}_2 \tilde{L}_1^\top & O_{\ell n} & \cdots & \tilde{L}_2 \tilde{L}_f^\top\\
\vdots & \vdots  & \ddots & \vdots\\
\tilde{L}_f \tilde{L}_1^\top & \tilde{L}_f \tilde{L}_2^\top & \cdots & O_{\ell n}
\end{pmatrix},\\
A_3+A_4&=(J_f-I_f)\otimes J_{\ell} \otimes J_n. 
\end{align*}

The following is the main result of this subsection. 
\begin{theorem}\label{thm:asclass5}
Let $H$ be a  balancedly splittable Hadamard matrix of order $n$ with $\sum_{i=1}^\ell C_i=\ell I_n+a A+b(J_n-A-I_n)$ where $A$ is the adjacency matrix of a regular graph, and $L_1,\ldots,L_f$ be mutually UFS Latin square on $\{1,\ldots,\ell\}$.  
Then the set of matrices $\{A_0,A_1,\ldots,A_5\}$ defined above is a symmetric association scheme with $5$-classes. 
\end{theorem}
\begin{proof}
It is easy to see that $A_0=I_{f \ell n}$, $A_i$'s are disjoint symmetric $(0,1)$-matrices such that $\sum_{i=0}^5 A_i=J_{f \ell n}$, and each $A_i$ is symmetric. 
Let $\mathcal{A}=\textrm{span}_{\mathbb{R}}\{A_0,A_1,\ldots,A_5\}$. 
 
First it can be shown that $\textrm{span}_{\mathbb{R}}\{A_0,A_1,A_2,A_3+A_4,A_5\}$ is closed under the matrix multiplication. 
Next we show that the products $A_i(A_3-A_4)$ for $i\in\{1,2,5\}$ and $(A_3-A_4)^2$ are linear combinations of $A_0,A_1,\ldots,A_5$, from which (4) in the definition of the association scheme follows. 
The equation $A_5(A_3-A_4)=O_{f\ell n}$ can be shown, and the cases for $A_i(A_3-A_4)$ for $i\in\{1,2\}$ follow from the following. 

Since $(I_{\ell}\otimes J_n) \tilde{L}_j=O_{\ell n}$ for each $j$, we have that $(A_0+A_1+A_2)(A_3-A_4)=O_{f\ell n}$. 
Therefore $(A_1+A_2)(A_3-A_4)=-A_3+A_4$.

Since $(\sum_{i=1}^{\ell}I_{\ell}\otimes C_i) \tilde{L}_j=\tilde{L}_j$ for each $j$, we have that $(aA_1+bA_2)(A_3-A_4)=(n-\ell)(A_3-A_4)$.



Finally from Lemmas~\ref{lem:L}, \ref{lem:SL} it follows that  
\begin{align*}
(A_3-A_4)^2
=n(f-1)(\ell A_0+a A_1+b A_2)+n(f-2)(A_3-A_4). 
\end{align*}
This completes the proof. 
\end{proof}

Then the eigenmatrices $P$ and $Q$ are as follows:
\begin{align*}
P&=\left(
\begin{array}{cccccc}
1 & \frac{\ell (n-\ell-1) n}{(n-1) a^2+2 \ell a+\ell (n-\ell)} & \frac{(\ell+a (n-1))^2}{(n-1) a^2+2 \ell a+\ell (n-\ell)} & \frac{1}{2} (f-1) \ell n & \frac{1}{2} (f-1) \ell n & (\ell-1) n \\
 1 & \frac{a n (n-\ell-1)}{(n-1) a^2+2 \ell a+\ell (n-\ell)} & -\frac{(\ell+a (n-1)) (a-\ell+n)}{(n-1) a^2+2 \ell a+\ell (n-\ell)} & \frac{1}{2} (f-1) n & \frac{1}{2} (n-f n) & 0 \\
 1 & -\frac{(a+1) \ell n}{(n-1) a^2+2 \ell a+\ell (n-\ell)} & -\frac{(a-\ell) (\ell+a (n-1))}{(n-1) a^2+2 \ell a+\ell (n-\ell)} & 0 & 0 & 0 \\
 1 & \frac{\ell (\ell-n+1) n}{-(n-1) a^2-2 \ell a+\ell (\ell-n)} & \frac{(\ell+a (n-1))^2}{(n-1) a^2+2 \ell a+\ell (n-\ell)} & 0 & 0 & -n \\
 1 & \frac{a n (-\ell+n-1)}{(n-1) a^2+2 \ell a+\ell (n-\ell)} & -\frac{(\ell+a (n-1)) (a-\ell+n)}{(n-1) a^2+2 \ell a+\ell (n-\ell)} & -\frac{n}{2} & \frac{n}{2} & 0 \\
 1 & \frac{\ell (\ell-n+1) n}{-(n-1) a^2-2 \ell a+\ell (\ell-n)} & \frac{(\ell+a (n-1))^2}{(n-1) a^2+2 \ell a+\ell (n-\ell)} & -\frac{1}{2} \ell n & -\frac{1}{2} \ell n & (\ell-1) n \\\end{array}
\right),\\ \displaybreak[0]\\
Q&=\left(
\begin{array}{cccccc}
 1 & \ell^2 & f \ell (n-\ell-1) & f (\ell-1) & (f-1) \ell^2 & f-1 \\
 1 & a \ell & -(a+1) f \ell & f (\ell-1) & a (f-1) \ell & f-1 \\
 1 & \frac{\ell^2 (-a+\ell-n)}{\ell+a (n-1)} & \frac{f (a-\ell) \ell (\ell-n+1)}{\ell+a (n-1)} & f (\ell-1) & -\frac{(f-1) \ell^2 (a-\ell+n)}{\ell+a (n-1)} & f-1 \\
 1 & \ell & 0 & 0 & -\ell & -1 \\
 1 & -\ell & 0 & 0 & \ell & -1 \\
 1 & 0 & 0 & -f & 0 & f-1 \\
\end{array}
\right).
\end{align*}

\subsection{Symmetric association schemes with $6$-classes}\label{sec:as6}
Let $L_1,\ldots,L_f$ be mutually UFS Latin squares on $\{0,1,\ldots,\ell\}$ with constant diagonal $0$. 
We now construct a symmetric association scheme with $6$-classes from a balancedly splittable Hadamard matrix and mutually UFS Latin squares. 
Consider the Gram matrix $G$ of the row vectors of $\tilde{L}_i$ ($i\in\{1,\ldots,f\}$) defined by $G=(\tilde{L}_i \tilde{L}_j^\top)_{i,j=1}^{f}$

The entries of $G$ are $\{n\ell, n a, n b ,\pm n,0\}$. Decompose the matrix $G$ into its entries as 
\begin{align*}
G=n\ell A_0+n(aA_1+bA_2)+n(A_3-A_4). 
\end{align*}
Then the disjoint $(0,1)$-matrices $A_i$ ($i\in\{0,1,\ldots,4\}$) satisfy $\sum_{i=0}^4 A_i=J_{f(\ell+1)n}-(I_f\otimes (J_{\ell+1}-I_{\ell+1})\otimes J_n+(J_f-I_f)\otimes I_{\ell+1} \otimes J_n)$. 
We now define 
\begin{align*}
A_5=I_f\otimes (J_{\ell+1}-I_{\ell+1})\otimes J_n,\quad A_6=(J_f-I_f)\otimes I_{\ell+1} \otimes J_n.
\end{align*}
Note that $A_1=I_f\otimes I_{\ell+1}\otimes A, A_2=I_f\otimes I_{\ell+1}\otimes (J_n-A-I_n)$ and
\begin{align*}
A_3-A_4&=\frac{1}{n}\begin{pmatrix}
O_{(\ell+1)n} & \tilde{L}_1 \tilde{L}_2^\top & \cdots & \tilde{L}_1 \tilde{L}_f^\top\\
\tilde{L}_2 \tilde{L}_1^\top & O_{(\ell+1)n} & \cdots & \tilde{L}_2 \tilde{L}_f^\top\\
\vdots & \vdots  & \ddots & \vdots\\
\tilde{L}_f \tilde{L}_1^\top & \tilde{L}_f \tilde{L}_2^\top & \cdots & O_{(\ell+1)n}
\end{pmatrix},\\
A_3+A_4&=(J_f-I_f)\otimes(J_{\ell+1}-I_{\ell+1}) \otimes J_n. 
\end{align*}

The following is the main result of this subsection. 
\begin{theorem}
Let $H$ be a balancedly splittable Hadamard matrix of order $n$ with $\sum_{i=1}^\ell C_i=\ell I_n+a A+b(J_n-A-I_n)$ where $A$ is the adjacency matrix of a regular graph, and $L_1,\ldots,L_f$ be mutually UFS Latin square on $\{0,1,\ldots,\ell\}$ with constant diagonal $0$.  
Then the set of matrices $\{A_0,A_1,\ldots,A_6\}$ defined above  is a symmetric association scheme with $6$-classes. 
\end{theorem}
\begin{proof}
The proof is similar to that of Theorem~\ref{thm:asclass5}. 
%
\end{proof}

Then the eigenmatrices $P$ and $Q$ are as follows:
\begin{align*}
P&=\left(
\begin{array}{ccccccc}
 1 & \frac{\ell (\ell-n+1) n}{-(n-1) a^2-2 \ell a+\ell (\ell-n)} & \frac{(\ell+a (n-1))^2}{(n-1) a^2+2 \ell a+\ell (n-\ell)} & \frac{1}{2} (f-1) \ell n & \frac{1}{2} (f-1) \ell n & \ell n & (f-1) n \\
 1 & \frac{\ell (\ell-n+1) n}{-(n-1) a^2-2 \ell a+\ell (\ell-n)} & \frac{(\ell+a (n-1))^2}{(n-1) a^2+2 \ell a+\ell (n-\ell)} & -\frac{1}{2} (f-1) n & -\frac{1}{2} (f-1) n & -n & (f-1) n \\
 1 & \frac{a n (-\ell+n-1)}{(n-1) a^2+2 \ell a+\ell (n-\ell)} & -\frac{(\ell+a (n-1)) (a-\ell+n)}{(n-1) a^2+2 \ell a+\ell (n-\ell)} & \frac{1}{2} (f-1) n & \frac{1}{2} (n-f n) & 0 & 0 \\
 1 & -\frac{(a+1) \ell n}{(n-1) a^2+2 \ell a+\ell (n-\ell)} & -\frac{(a-\ell) (\ell+a (n-1))}{(n-1) a^2+2 \ell a+\ell (n-\ell)} & 0 & 0 & 0 & 0 \\
 1 & \frac{a n (-\ell+n-1)}{(n-1) a^2+2 \ell a+\ell (n-\ell)} & -\frac{(\ell+a (n-1)) (a-\ell+n)}{(n-1) a^2+2 \ell a+\ell (n-\ell)} & -\frac{n}{2} & \frac{n}{2} & 0 & 0 \\
 1 & \frac{\ell (\ell-n+1) n}{-(n-1) a^2-2 \ell a+\ell (\ell-n)} & \frac{(\ell+a (n-1))^2}{(n-1) a^2+2 \ell a+\ell (n-\ell)} & \frac{n}{2} & \frac{n}{2} & -n & -n \\
 1 & \frac{\ell (\ell-n+1) n}{-(n-1) a^2-2 \ell a+\ell (\ell-n)} & \frac{(\ell+a (n-1))^2}{(n-1) a^2+2 \ell a+\ell (n-\ell)} & -\frac{1}{2} (\ell n) & -\frac{1}{2} (\ell n) & \ell n & -n \\
\end{array}
\right),\\
Q&=\left(
\begin{array}{ccccccc}
 1 & \ell & \ell (\ell+1) & -f (\ell+1) (\ell-n+1) & (f-1) \ell (\ell+1) & (f-1) \ell & f-1 \\
 1 & \ell & a (\ell+1) & -(a+1) f (\ell+1) & a (f-1) (\ell+1) & (f-1) \ell & f-1 \\
 1 & \ell & \frac{\ell (\ell+1) (-a+\ell-n)}{\ell+a (n-1)} & \frac{f (a-\ell) (\ell+1) (\ell-n+1)}{\ell+a (n-1)} & -\frac{(f-1) \ell (\ell+1) (a-\ell+n)}{\ell+a (n-1)} & (f-1) \ell & f-1 \\
 1 & -1 & \ell+1 & 0 & -\ell-1 & 1 & -1 \\
 1 & -1 & -\ell-1 & 0 & \ell+1 & 1 & -1 \\
 1 & -1 & 0 & 0 & 0 & 1-f & f-1 \\
 1 & \ell & 0 & 0 & 0 & -l & -1 \\
\end{array}
\right).
\end{align*}

\section*{Acknowledgments.}
The authors are grateful to an anonymous referee for pointing out some calculation errors and for many suggestions which has improved the presentation of the paper. 
The authors acknowledge very useful conversation with Darcy Best.
Hadi Kharaghani is supported by the Natural Sciences and 
Engineering  Research Council of Canada (NSERC).  Sho Suda is supported by JSPS KAKENHI Grant Number 15K21075 and 18K03395.

\end{document}